\documentclass[11pt,a4paper]{amsart}
\usepackage{latexsym}
\usepackage{amssymb}
\usepackage[all]{xy}

\newtheorem{theorem}{Theorem}[section]
\newtheorem{corollary}[theorem]{Corollary}
\newtheorem{lemma}[theorem]{Lemma}
\newtheorem{proposition}[theorem]{Proposition}
\newtheorem{definition}[theorem]{Definition}
\newtheorem{example}[theorem]{Example}
\newtheorem{remark}[theorem]{Remark}

\newcommand{\Ker}{{\rm ker}}

\begin{document}
\sloppy

\title[Stable short exact sequences]{When stable short exact sequences define an exact structure on an additive category}

\author{Septimiu Crivei}

\address{Faculty of Mathematics and Computer Science \\ ``Babe\c s-Bolyai" University \\ Str. Mihail Kog\u alniceanu 1
\\ 400084 Cluj-Napoca, Romania} \email{crivei@math.ubbcluj.ro}


\begin{abstract} Rump has recently showed the existence of a unique maximal Quillen exact structure on any additive
category. We study when this is given by the stable short exact sequences, i.e. kernel-cokernel pairs
consisting of a semi-stable kernel and a semi-stable cokernel. 
\end{abstract}

\date{September 21, 2013}

\thanks{The author would also like to thank Mike Prest for illuminating discussions leading to Example \ref{e:SR}.}

\maketitle            

\section{Introduction}

Exact categories provide a suitable setting for developing homological algebra beyond abelian categories, and have
important applications in different fields, such as algebraic $K$-theory, algebraic geometry and topology, algebraic and
functional analysis etc., where many relevant categories have a poorer algebraic structure than an abelian category. The
most important concept of exact category in the additive case has crystalized in the work of Quillen on algebraic
$K$-theory \cite{Q}, and has been simplified by Keller \cite{Keller}. This extends different previous notions such
as those of Heller \cite{Heller} and Yoneda \cite{Y}, to mention only the closest ones. A Quillen exact category is an
additive category endowed with a distinguished class of kernel-cokernel pairs, called \emph{conflations}, satisfying
certain axioms (see Definition \ref{d:exact}). Several recent papers present exhaustive accounts on Quillen exact
categories \cite{Buhler,FS}, and rebring into the main attention their power, visible in some new applications, such as
those in functional analysis \cite{FS,SW}, model structures \cite{G} or approximation theory \cite {MS}.  

Every additive category has a smallest exact structure, whose conflations are the split exact sequences. It is natural
to wonder if there exists, and then which is, the greatest exact structure on an arbitrary additive category. More
important than its maximality it is the possibility to have a natural exact structure on any additive category, apart
from the trivial one given by the split exact sequences, with respect to which to develop homological algebra. Since
every conflation is a kernel-cokernel pair, a first candidate for the class of conflations should obviously be the
maximal possible class in the definition of an exact category, namely that of all kernel-cokernel pairs. While this is
the right choice for abelian categories, or even for quasi-abelian categories (i.e. additive categories with kernels and
cokernels such that pushouts preserve kernels and pullbacks preserve cokernels) \cite{Gruson,Rump01,Sch}, it is no
longer suitable for the more general setting of preabelian categories (i.e. additive categories with kernels and
cokernels). 

It has already been observed by Richman and Walker \cite[p.~522]{RW} that the class of all kernels
(cokernels) in a preabelian category is not closed under pushouts (pullbacks), and so the class of all kernel-cokernel
pairs cannot define an exact structure in this setting. In order to study homological algebra in preabelian
categories they introduced \emph{semi-stable} kernels (cokernels) as those kernels (cokernels) which are
preserved by pushouts (pullbacks), and the \emph{stable} short exact sequences as those kernel-cokernel pairs
consisiting of a semi-stable kernel and a semi-stable cokernel. Recently, Sieg and Wegner have made use of these
concepts, and showed that the class of stable short exact sequences defines the unique maximal exact structure on any
preabelian category \cite[Theorem~3.3]{SW}. A natural extension of the definition of semi-stable kernels and semi-stable
cokernels from a preabelian category to an arbitrary additive category allowed the generalization of the above result in
\cite[Theorem~3.5]{C12}, which shows that the class of stable short exact sequences defines the unique maximal exact
structure on any weakly idempotent complete additive category. Note that many classes of additive categories, such as
the accessible categories (which have a natural exact structure consisting of the pure exact sequences) \cite{Prest} and
the triangulated categories (which only have the trivial exact structure), are weakly idempotent complete. But there
are additive categories (even exact) which are not weakly idempotent complete, for instance any category of free
modules in which there exist projective modules which are not free (e.g., see \cite[p.~2894]{G}).

The remaining challenge was to determine a greatest exact structure on any additive category. A step towards that
direction has recently been made by Rump, which shows that there exists a greatest exact structure on any additive
category \cite{Rump11}. His interesting approach uses a new concept of
one-sided exact category (also, see \cite{BC}), by constructing the maximal left exact structure and the maximal
right exact structure, and then deducing the existence of the greatest exact structure. Nevertheless, the question
whether this is defined by the class of stable kernel-cokernel pairs has still remained open. The main obstacle is 
to prove that the semi-stable kernels and the semi-stable cokernels satisfy Quillen's ``obscure axiom'' (see
\cite{C12,Rump11}). 

In the present paper we establish an equivalent condition for the stable short exact sequences to define an exact structure on 
an additive category, and we shall show that this is not always checked. 
Our approach uses essentially the property that every additive category has an additive idempotent (or Karoubian) completion 
\cite[p.~75]{K}, which in turn is weakly idempotent complete \cite{Buhler}. 
The technique is to employ the canonical fully faithful functor 
$H:\mathcal{C}\to \widehat{\mathcal{C}}$ between an additive category $\mathcal{C}$ and its idempotent completion 
$\widehat{\mathcal{C}}$ in order to transfer properties. One of the main ingredients is to prove that 
$H$ preserves stable short exact sequences. Then we show that the stable short exact sequences define an exact structure 
on an additive category $\mathcal{C}$ if and only if $\mathcal{C}$ is stable under pushouts and pullbacks for 
$(\widehat{\mathcal{C}},H)$ in the sense of the forthcoming Definition \ref{d:stablePOPB}. The latter condition is 
usually easier to check, and we illustrate this in two relevant situations. We also give 
some applications to categories of chain complexes and projective spectra. 

Our characterization is inspired by that of an {\it exact full subcategory} $\mathcal{C}'$ of an additive category $\mathcal{C}$ 
(in the sense that the restriction of an exact structure $\mathcal{E}$ from $\mathcal{C}$ to the full subcategory $\mathcal{C}'$ is still 
an exact structure) \cite[Theorem~2.6]{DS}. As already noted by Dierolf and Sieg \cite{DS}, any extension-closed subcategory of 
an exact category is an exact subcategory in the above sense. But there are exact subcategories 
$\mathcal{C}'$ of some exact category $\mathcal{C}$ which are not extension-closed in $\mathcal{C}$ (e.g., see \cite[Remark~2.3]{DS}, 
which argues that the full subcategory $LCS$ of locally convex spaces of the category $TVS$ of topological vector spaces with 
the exact structure given by the short topologically exact sequences is an exact subcategory which is not extension-closed in $TVS$), 
even if the Gabriel-Quillen embedding theorem allows one to realize them as extension-closed subcategories of some abelian categories.

\section{Idempotent completion}

An additive category is called \emph{idempotent complete} if every idempotent morphism has a kernel
\cite[D\'efinition~1.2.1]{K}. A remarkable result states that every additive category has an \emph{idempotent
completion} (also called \emph{Karoubian completion}) \cite[Lemme~1.2.2]{K}. More precisely, for every additive category
$\mathcal{C}$, there exists an idempotent complete additive category $\widehat{\mathcal{C}}$ and a fully faithful
additive functor $H:\mathcal{C}\to \widehat{\mathcal{C}}$. 

The category $\widehat{\mathcal{C}}$ has as objects the pairs
$(A,p)$, where $A$ is an object of $\mathcal{C}$ and $p:A\to A$ is an idempotent morphism in $\mathcal{C}$, and as 
morphisms between two objects $(A,p)$ and $(B,q)$ of $\widehat{\mathcal{C}}$ the morphisms $f:A\to B$ in
$\mathcal{C}$ such that $f=qfp$. The biproduct in $\widehat{\mathcal{C}}$ is given by $(A,p)\oplus (B,q)=(A\oplus
B,p\oplus q)$. The functor $H:\mathcal{C}\to \widehat{\mathcal{C}}$ is defined by $H(A)=(A,1_A)$ on
objects $A$ of $\mathcal{C}$, and by $H(f)=f$ on morphisms $f$ in $\mathcal{C}$ (also see \cite[Section~6]{Buhler}).
We denote by ${\rm Im}(H)$ the essential image of $H$.

The next remark will be important in what follows.

\begin{remark} \rm (i) Every idempotent complete additive category is weakly idempotent complete, in the sense that
every section has a cokernel, or equivalently, every retraction has a kernel \cite[Lemma~7.1]{Buhler}. 

(ii) Every object of the idempotent completion $(\widehat{\mathcal{C}},H)$ of an additive category $\mathcal{C}$ is a
direct summand of an object in ${\rm Im}(H)$. Indeed, for every object $(A,p)$ of $\widehat{\mathcal{C}}$, the
morphism $p:(A,p)\to (A,1_A)$ splits, whence $(A,p)$ is a direct summand of $H(A)=(A,1_A)$, because
$\widehat{\mathcal{C}}$ is weakly idempotent complete (e.g., see \cite[Remark~7.4]{Buhler}). The identity morphism ${\rm
id}_A$ of $(A,p)$ is $p$.
\end{remark}

\begin{lemma} \label{l:presrefl} Let $\mathcal{C}$ be an additive category and let $(\widehat{\mathcal{C}},H)$ be its
idempotent completion. Then $H$ preserves and reflects pullbacks and pushouts.
\end{lemma}

\begin{proof} Since $H$ is fully faithful, it reflects pullbacks and pushouts \cite[Chapter~II,~Theorem~7.1]{M}. We show
that $H$ preserves pullbacks, the preservation of pushouts being dual. 
Let $(B'=B\times_CC',g,d')$ be a pullback of some morphisms $d:B\to C$ and $h:C'\to C$ in $\mathcal{C}$. Let $(D,p)$ be
an object of $\widehat{\mathcal{C}}$ for some object $D$ and idempotent morphism $p:D\to D$ in $\mathcal{C}$, and let
$\alpha:(D,p)\to H(B)$, $\beta:(D,p)\to H(C')$ be morphisms in $\widehat{\mathcal{C}}$ such that
$H(d)\alpha=H(h)\beta$. Then $\alpha:D\to B$, $\beta:D\to C'$, $\alpha=\alpha p$ and $\beta=\beta p$.
By the pullback property there is a unique morphism $w:D\to B'$ in $\mathcal{C}$ such that $gw=\alpha$ and
$d'w=\beta$. Then the morphism $wp:(D,p)\to H(B')$ in $\widehat{\mathcal{C}}$ satisfies the equalities
$H(g)wp=\alpha$ and $H(d')wp=\beta$. If $w':(D,p)\to H(B')$ is another morphism in $\widehat{\mathcal{C}}$ such that
$H(g)w'=\alpha$ and $H(d')w'=\beta$, then $w':D\to B'$ and $w'=w'p$ in $\mathcal{C}$. By the pullback property
of $d$ and $h$, it follows that $w'=w$, and so $w'=w'p=wp$. Hence $(H(B'),H(g),H(d'))$ is a pullback of $H(d)$ and
$H(h)$ in $\widehat{\mathcal{C}}$.
\end{proof}

The notion of stable short exact sequence was introduced in \cite{RW} for preabelian categories, and generalized to
arbitrary categories in \cite{C12} as follows. 

\begin{definition} \label{d:ss} \rm Let $\mathcal{C}$ be a category. A cokernel $d:B\to C$ in $\mathcal{C}$
is called a \emph{semi-stable cokernel} if there exists a pullback $(B'=B\times_CC',g,d')$ of $d$ along any morphism
$h:C'\to C$: 
\[\SelectTips{cm}{}
\xymatrix{
B' \ar[d]_g \ar[r]^{d'} & C' \ar[d]^h \\ 
B \ar[r]^{d} & C  
}\] 
such that $d':B'\to C'$ is a cokernel. The notion of \emph{semi-stable kernel} is defined dually. A short exact
sequence, i.e. a kernel-cokernel pair, $A\overset{i}\to B\overset{d}\to C$ is called \emph{stable} if $i$
is a semi-stable kernel and $d$ is a semi-stable cokernel. 
\end{definition}

We need two well-known result on pullbacks, whose duals for pushouts hold as well.

\begin{lemma}{\cite[Lemma~5.1]{Kelly}} \label{l:PB} Consider the following diagram in a category $\mathcal{C}$ such that
the squares are commutative and the right square is a pullback: 
\[\SelectTips{cm}{}
\xymatrix{
A' \ar[d]_f \ar[r]^{i'} & B' \ar[d]_g  \ar[r]^{d'} & C' \ar[d]^h \\ 
A \ar[r]^i & B\ar[r]^d & C 
}\]
Then the left square is a pullback if and only if so is the rectangle.
\end{lemma}

\begin{lemma}{\cite[Theorem~5]{RW}} \label{l:RW} Let $d:B\to C$ and $h:C'\to C$ be morphisms in a category $\mathcal{C}$
such that $d$ has a kernel $i:A\to B$, and the pullback of $d$ and $h$ exists. Then there exists a commutative diagram
in $\mathcal{C}$: 
\[\SelectTips{cm}{}
\xymatrix{
A \ar@{=}[d] \ar[r]^{i'} & B' \ar[d]_g \ar[r]^{d'} & C' \ar[d]^h \\ 
A \ar[r]^i & B \ar[r]^d & C  
}\] 
in which the right square is a pullback and $i':A\to B'$ is the kernel of $d'$.
\end{lemma}

Let us note some useful remarks, which will be freely used together with their dual versions.

\begin{remark} \label{r:rem} \rm (i) Every semi-stable cokernel is the cokernel of its kernel by
\cite[Remark~2.5]{C12}. 

(ii) The pullback of a semi-stable cokernel along any morphism exists and is again a semi-stable cokernel by Lemma
\ref{l:PB}.

(iii) Every projection onto a direct summand is a semi-stable cokernel by \cite[Corollary~3.3]{C12}. In particular,
every isomorphism is a semi-stable cokernel.
\end{remark}

The following result is one of the keys for establishing our main theorem.

\begin{proposition} \label{p:main} Let $\mathcal{C}$ be an additive category and let $(\widehat{\mathcal{C}},H)$ be its
idempotent completion. Then $H$ preserves stable short exact sequences.
\end{proposition}

\begin{proof} Let $A\stackrel{i}\to B\stackrel{d}\to C$ be a stable short exact sequence in $\mathcal{C}$. Then
$H(A)\stackrel{H(i)}\to H(B)\stackrel{H(d)}\to H(C)$ is a kernel-cokernel pair in $\widehat{\mathcal{C}}$ by Lemma
\ref{l:presrefl}. We show that $H(d)$ is a semi-stable cokernel, the proof that $H(i)$ is a semi-stable
kernel following in a dual manner. 

Let $\gamma:Z\to H(C)$ be a morphism in $\widehat{\mathcal{C}}$. Then $Z$ is a direct
summand of an object in ${\rm Im}(H)$, say $Z\oplus Z'=H(C')$ for some objects $Z'$ of $\widehat{\mathcal{C}}$ and $C'$
of $\mathcal{C}$. The composition $\left [\begin{smallmatrix} \gamma & 0 \end{smallmatrix}\right ]=\gamma \left
[\begin{smallmatrix} 1 & 0 \end{smallmatrix}\right ]:H(C')=Z\oplus Z'\to H(C)$ must be of the form $H(h)$ for some
morphism $h:C'\to C$ in $\mathcal{C}$. Since $d$ is a semi-stable cokernel, we have the following commutative diagram
in $\mathcal{C}$:
\[\SelectTips{cm}{}
\xymatrix{
A \ar@{=}[d] \ar[r]^{i'} & B' \ar[d]_g \ar[r]^{d'} & C' \ar[d]^h \\ 
A \ar[r]^{i} & B \ar[r]^{d} & C  
}\]
in which the right square is a pullback, and the rows are kernel-cokernel pairs (see Lemma \ref{l:RW}). The
morphism $H(d'):H(B')\to Z\oplus Z'$ is of the form $\left [\begin{smallmatrix} u \\ v \end{smallmatrix}\right ]$ for
some morphisms $u:H(B')\to Z$ and $v:H(B')\to Z'$ in $\widehat{\mathcal{C}}$. Then we have the following commutative
diagram in $\widehat{\mathcal{C}}$:
\[\SelectTips{cm}{}
\xymatrix{
H(A) \ar@{=}[d] \ar[r]^-{H(i')} & H(B') \ar[d]_{H(g)} \ar[r]^-{\left [\begin{smallmatrix} u \\ v \end{smallmatrix}\right
]} & Z\oplus Z' \ar[d]^{\left [\begin{smallmatrix} \gamma & 0 \end{smallmatrix}\right ]} \\ H(A) \ar[r]_-{H(i)} & H(B)
\ar[r]_-{H(d)} & H(C) 
}\] 
in which the right square is a pullback and the rows are kernel-cokernel pairs by Lemmas \ref{l:presrefl} and
\ref{l:RW}. Since $H(d)0=\left [\begin{smallmatrix} \gamma&0 \end{smallmatrix}\right ]\left [\begin{smallmatrix} 0\\1
\end{smallmatrix}\right ]$, by the pullback property there is a unique morphism $\varepsilon:Z'\to H(B')$ such that
$H(g)\varepsilon=0$ and $\left [\begin{smallmatrix} u \\ v \end{smallmatrix}\right ]\varepsilon=\left
[\begin{smallmatrix} 0\\1 \end{smallmatrix}\right ]$. In particular, $v\varepsilon={\rm id}_{Z'}$, and so $v$ is a
retraction. Since $\widehat{\mathcal{C}}$ is weakly idempotent complete, $v$ has a kernel, say $j:K\to H(B')$. Then
$H(d)H(g)j=\gamma uj$. 

We claim that the following commutative square:
\[\SelectTips{cm}{}
\xymatrix{
K\ar[r]^{uj} \ar[d]_{H(g)j} & Z \ar[d]^{\gamma} \\ 
H(B) \ar[r]_{H(d)} & H(C) }
\]
is a pullback of $H(d)$ and $\gamma$ in $\widehat{\mathcal{C}}$. To this end, let $a:D\to H(B)$ and $b:D\to Z$ be
morphisms such that $H(d)a=\gamma b$. Then $H(d)a=\left [\begin{smallmatrix} \gamma&0 \end{smallmatrix}\right ]\left
[\begin{smallmatrix} b\\0 \end{smallmatrix}\right ]$, hence the pullback square of $H(d)$ and $\left
[\begin{smallmatrix} \gamma&0 \end{smallmatrix}\right ]$ yields the existence of a unique morphism $\delta:D\to H(B')$
such that $H(g)\delta=a$ and $\left [\begin{smallmatrix} u \\ v \end{smallmatrix}\right ]\delta=\left
[\begin{smallmatrix} b\\0 \end{smallmatrix}\right ]$. Since $v\delta=0$ and $j=\Ker(v)$, there
is a unique morphism $w:D\to K$ such that $\delta=jw$. Then we have $H(g)jw=a$ and $ujw=b$. For uniqueness, suppose
that there is another morphism $w':D\to K$ such that $H(g)jw'=a$ and $ujw'=b$. It follows that 
$H(g)(jw-jw')=0$ and $\left [\begin{smallmatrix} u \\ v \end{smallmatrix}\right ](jw-jw')=\left [\begin{smallmatrix} 0
\\ 0 \end{smallmatrix}\right ]$, whence $jw-jw'=0$ by the pullback property of $H(d)$ and $\left
[\begin{smallmatrix} \gamma&0 \end{smallmatrix}\right ]$. Then $w=w'$, because $j$ is a monomorphism. Thus $H(d)$
and $\gamma$ have the claimed pullback in $\widehat{\mathcal{C}}$.

Since $H(d)H(g)=\gamma \left [\begin{smallmatrix} 1&0 \end{smallmatrix}\right ]\left [\begin{smallmatrix} u
\\ v \end{smallmatrix}\right ]$, the pullback square of $H(d)$ and $\gamma$ yields the existence of a unique morphism
$r:H(B')\to K$ such that $H(g)jr=H(g)$ and $ujr=\left [\begin{smallmatrix} 1&0 \end{smallmatrix}\right ]\left
[\begin{smallmatrix} u \\ v \end{smallmatrix}\right ]$. Denote $k=rH(i'):H(A)\to K$. Then we have the following
commutative diagram in $\widehat{\mathcal{C}}$:
\[\SelectTips{cm}{}
\xymatrix{
H(A) \ar@{=}[d] \ar[r]^-{H(i')} & H(B') \ar[d]_r \ar[r]^-{\left [\begin{smallmatrix} u \\ v \end{smallmatrix}\right ]} &
Z\oplus Z' \ar[d]^{\left [\begin{smallmatrix} 1 & 0 \end{smallmatrix}\right ]} \\ H(A) \ar[r]^-k & K\ar[r]^{uj}
\ar[d]_{H(g)j} & Z \ar[d]^{\gamma} \\ & H(B) \ar[r]_{H(d)} & H(C)
}\] 
in which the upper row is a kernel-cokernel pair. The vertical rectangle and the lower square are pullbacks, hence so
is the right-upper square by Lemma \ref{l:PB}. Since $uj{\rm id}_K=\left [\begin{smallmatrix} 1 & 0
\end{smallmatrix}\right ]
\left [\begin{smallmatrix} uj \\ 0 \end{smallmatrix}\right ]$, the pullback property of $uj$ and $\left
[\begin{smallmatrix} 1 & 0 \end{smallmatrix}\right ]$ yields the existence of a unique morphism $j':K\to H(B')$ such
that $rj'={\rm id}_K$ and $\left [\begin{smallmatrix} u \\ v \end{smallmatrix}\right ]j'=\left [\begin{smallmatrix} uj
\\ 0 \end{smallmatrix}\right ]$. Then the equality $H(g)jr=H(g)$ implies that $H(g)j=H(g)j'$, and so $H(g)(j-j')=0$. We
also have $\left [\begin{smallmatrix} u \\ v \end{smallmatrix}\right ](j-j')=\left [\begin{smallmatrix} 0 \\ 0
\end{smallmatrix}\right ]$. Now the pullback property of $H(d)$ and $\left [\begin{smallmatrix} \gamma & 0
\end{smallmatrix}\right ]$ implies that $j-j'=0$, and so $rj=rj'={\rm id}_K$. 

The commutative right-upper square in the previous diagram implies that $ujr$ is an epimorphism, hence so is
$uj$. Also, we have $ujk=0$. We claim that $uj={\rm coker}(k)$. To this end, let $t:K\to T$ be a morphism such that
$tk=0$. Then the equality $trH(i')=0$ and the fact that $\left [\begin{smallmatrix} u \\ v \end{smallmatrix}\right
]={\rm coker}(H(i'))$ imply the existence of a unique morphism $\left [\begin{smallmatrix} z & z'
\end{smallmatrix}\right ]:Z\oplus Z'\to T$ such that $\left [\begin{smallmatrix} z & z' \end{smallmatrix}\right ]\left
[\begin{smallmatrix} u \\ v \end{smallmatrix}\right ]=tr$, that is, $zu+z'v=tr$. It follows that $zuj+z'vj=trj$, that
is, $zuj=t$. The uniqueness of the morphism $z:Z\to T$ such that $zuj=t$ follows because $uj$ is an epimorphism. Hence
$uj$ is a cokernel. Consequently, $H(d)$ is a semi-stable cokernel in $\widehat{\mathcal{C}}$.
\end{proof}

\section{Maximal exact structure}

We consider the following concept of exact category given by Quillen \cite{Q}, as simplified by Keller \cite{Keller}.

\begin{definition} \label{d:exact} \rm By an \emph{exact category} we mean an additive category $\mathcal{C}$ endowed
with a distinguished class $\mathcal{E}$ of short exact sequences satisfying the axioms $[E0]$, $[E1]$, $[E2]$ and
$[E2^{\rm op}]$ below. The short exact sequences in $\mathcal{E}$ are called \emph{conflations}, while the kernels and
cokernels appearing in such exact sequences are called \emph{inflations} (denoted by $\rightarrowtail$) and
\emph{deflations} (denoted by $\twoheadrightarrow$) respectively. \vskip2mm

$[E0]$ The identity morphism $1_0:0\to 0$ is a deflation.

$[E1]$ The composition of two deflations is again a deflation.

$[E2]$ The pullback of a deflation along an arbitrary morphism exists and is again a deflation.

$[E2^{\rm op}]$ The pushout of an inflation along an arbitrary morphism exists and is again an inflation.
\end{definition}

We should note that the duals of the axioms $[E0]$ and $[E1]$ on inflations as well as both sides of Quillen's ``obscure
axiom'' hold in any such exact category \cite{Keller}. 

Following \cite[Definition~2.4]{DS}, we introduce the following terminology.

\begin{definition} \label{d:stablePOPB} \rm Let $\mathcal{C}$ be an additive category and let $(\widehat{\mathcal{C}},H)$ be its idempotent
completion. Consider an exact structure $\mathcal{E}$ on $\widehat{\mathcal{C}}$. Then $\mathcal{C}$ is called
\emph{stable for pushouts for $(\widehat{\mathcal{C}},\mathcal{E})$} if for every conflation $H(A)\stackrel{H(i)}\rightarrowtail
H(B)\stackrel{H(d)}\twoheadrightarrow H(C)$ and every pushout diagram 
\[\SelectTips{cm}{}
\xymatrix{
H(A) \ar[d]_{H(f)} \ar@{>->}[r]^-{H(i)} & H(B) \ar[d]^w \\ 
H(A') \ar@{>->}[r]^-{u} & S  
}\]
in $\widehat{\mathcal{C}}$, there exists an object $B'$ of $\mathcal{C}$ such that $S\cong H(B')$.

Dually, one defines $\mathcal{C}$ to be \emph{stable for pullbacks for $(\widehat{\mathcal{C}},\mathcal{E})$}. 
\end{definition}

The following proposition extends \cite[Theorem~2.6]{DS}. 

\begin{proposition} \label{p:char} Let $\mathcal{C}$ be an additive category and let $(\widehat{\mathcal{C}},H)$ be its
idempotent completion. Consider an exact structure $\mathcal{E}$ on $\widehat{\mathcal{C}}$, and the distinguished
class $\mathcal{E}_0$ of kernel-cokernel pairs $A\rightarrowtail B\twoheadrightarrow C$ in $\mathcal{C}$ such that
the induced kernel-cokernel pair $H(A)\rightarrowtail H(B)\twoheadrightarrow H(C)$ is a conflation in
$\widehat{\mathcal{C}}$. Then $\mathcal{E}_0$ defines an exact structure on $\mathcal{C}$ if and only if $\mathcal{C}$
is stable under pushouts and pullbacks for $(\widehat{\mathcal{C}},\mathcal{E})$.  
\end{proposition}

\begin{proof} Suppose first that $\mathcal{E}_0$ defines an exact structure on $\mathcal{C}$. We only show that
$\mathcal{C}$ is stable under pushouts for $(\widehat{\mathcal{C}},\mathcal{E})$, the part for pullbacks being dual. 
Let $H(A)\stackrel{H(i)}\rightarrowtail H(B)\stackrel{H(d)}\twoheadrightarrow H(C)$ be a conflation and 
$H(f):H(A)\to H(A')$ a morphism in $\widehat{\mathcal{C}}$, and consider a pushout of $H(i)$ and $H(f)$. Then we have an 
induced commutative diagram in $\widehat{\mathcal{C}}$:
\[\SelectTips{cm}{}
\xymatrix{
H(A) \ar[d]_{H(f)} \ar@{>->}[r]^-{H(i)} & H(B) \ar[d]^w \ar@{->>}[r]^-{H(d)} & H(C) \ar@{=}[d] \\ 
H(A') \ar@{>->}[r]^-u & S \ar@{->>}[r]^-v & H(C) 
}\] in which the lower row is also a conflation. Since $A\stackrel{i}\rightarrowtail B\stackrel{d}\twoheadrightarrow C$
is a kernel-cokernel pair in $\mathcal{C}$ by Lemma \ref{l:presrefl}, it follows that it belongs to $\mathcal{E}_0$,
hence it is a conflation in $\mathcal{C}$. Then there exists the pushout of $i$ and $f$ in $\mathcal{C}$ and we have a
commutative diagram in $\mathcal{C}$:
\[\SelectTips{cm}{}
\xymatrix{
A \ar[d]_{f} \ar@{>->}[r]^-{i} & B \ar[d]^g \ar@{->>}[r]^-{d} & C \ar@{=}[d] \\ 
A' \ar@{>->}[r]^-{i'} & B' \ar@{->>}[r]^-{d'} & C 
}\] in which the lower row is also a conflation. It follows that 
\[\SelectTips{cm}{}
\xymatrix{
H(A) \ar[d]_{H(f)} \ar@{>->}[r]^-{H(i)} & H(B) \ar[d]^{H(g)} \\ 
H(A') \ar@{>->}[r]^-{H(i')} & H(B')  
}\] is a commutative square in $\widehat{\mathcal{C}}$. By the universal property of the pushout of $H(i)$ and $H(f)$
there exists a unique morphism $\lambda:S\to H(B')$ in $\widehat{\mathcal{C}}$ such that $\lambda u=H(i')$ and $\lambda
w=H(g)$. Then $H(d')\lambda u=0=vu$ and $H(d')\lambda w=H(d)=vw$, and again the universal property of the pushout of
$H(i)$ and $H(f)$ implies that $H(d')\lambda=v$. Hence the following diagram in $\widehat{\mathcal{C}}$ is commutative:
\[\SelectTips{cm}{}
\xymatrix{
H(A') \ar@{=}[d] \ar@{>->}[r]^-u & S \ar[d]^{\lambda} \ar@{->>}[r]^-v & H(C) \ar@{=}[d] \\ 
H(A') \ar@{>->}[r]^-{H(i')} & H(B') \ar@{->>}[r]^-{H(d')} & H(C) 
}\] Since $A'\stackrel{i'}\rightarrowtail B'\stackrel{d'}\twoheadrightarrow C$ is a conflation in $\mathcal{C}$, the
lower row of the above diagram is a conflation in $\mathcal{C}$ by the choice of $\mathcal{E}_0$. Now the Short Five
Lemma in $\widehat{\mathcal{C}}$ \cite[Corollary~3.2]{Buhler} implies that $\lambda$ is an isomorphism, and so $S\cong
H(B')$. This shows that $\mathcal{C}$ is stable under pushouts for $(\widehat{\mathcal{C}},\mathcal{E})$.
 
Conversely, suppose that $\mathcal{C}$ is stable under pushouts and pullbacks for $(\widehat{\mathcal{C}},\mathcal{E})$.
It is clear that the axiom $[E0]$ holds for $\mathcal{E}_0$. In order to show $[E2]$, let $A\stackrel{i}\rightarrowtail
B\stackrel{d}\twoheadrightarrow C$ be a kernel-cokernel pair in $\mathcal{E}_0$, and let $f:A\to A'$ be a morphism in
$\mathcal{C}$. Then $H(A)\stackrel{H(i)}\rightarrowtail H(B)\stackrel{H(d)}\twoheadrightarrow H(C)$ is a conflation in
$\widehat{\mathcal{C}}$, and we may consider the pushout of $H(i)$ and $H(f)$ in $\widehat{\mathcal{C}}$. Since
$\mathcal{C}$ is stable under pushouts for $(\widehat{\mathcal{C}},\mathcal{E})$, we may choose a pushout diagram of
the following form:
\[\SelectTips{cm}{}
\xymatrix{
H(A) \ar[d]_{H(f)} \ar@{>->}[r]^-{H(i)} & H(B) \ar[d]^{H(g)} \ar[r]^-{H(d)} & H(C) \ar@{=}[d] \\ 
H(A') \ar@{>->}[r]^-{H(i')} & H(B') \ar[r]^-{H(d')} & H(C) 
}\] in which the lower row is a conflation in $\widehat{\mathcal{C}}$. Using Lemma \ref{l:presrefl} we obtain a pushout
of $i$ and $f$ in $\mathcal{C}$ reflected from $\widehat{\mathcal{C}}$ in which the lower row is a kernel-cokernel pair
in $\mathcal{E}_0$, as required. Dually, $[E2^{\rm op}]$ holds for $\mathcal{E}_0$. 

Finally we need to prove $[E1]$. To this end, let $A\stackrel{f}\rightarrowtail B\stackrel{g}\twoheadrightarrow C$ and
$A'\stackrel{f'}\rightarrowtail C\stackrel{g'}\twoheadrightarrow C'$ be kernel-cokernel pairs in $\mathcal{E}_0$. Then
$H(A)\stackrel{H(f)}\rightarrowtail H(B)\stackrel{H(g)}\twoheadrightarrow H(C)$ and
$H(A')\stackrel{H(f')}\rightarrowtail H(C)\stackrel{H(g')}\twoheadrightarrow H(C')$ are conflations in
$\widehat{\mathcal{C}}$ by the choice of $\mathcal{E}_0$. By the axiom $[E1]$ for the exact structure $\mathcal{E}$, it
follows that $H(g'g)=H(g')H(g)$ is a deflation in $\widehat{\mathcal{C}}$. Since $\mathcal{C}$ is stable under pullbacks
for $(\widehat{\mathcal{C}},\mathcal{E})$, we may choose a pullback of $H(g)$ and $H(f')$ in $\widehat{\mathcal{C}}$ of
the form:
\[\SelectTips{cm}{}
\xymatrix{
H(K) \ar[d]_{H(k)} \ar@{->>}[r]^-{H(\alpha)} & H(A') \ar@{>->}[d]^{H(f')} \\ 
H(B) \ar@{->>}[r]^-{H(g)} & H(C)  
}\] Since $H(f')$ is a monomorphism in $\widehat{\mathcal{C}}$, so is $H(k)$. 
We claim that $H(k):H(K)\to H(B)$ is the kernel of $H(g'g)$ in $\widehat{\mathcal{C}}$. To this end, let $t:T\to H(B)$
be a morphism in $\widehat{\mathcal{C}}$ such that $H(g'g)t=0$. Since $H(f')$ is the kernel of $H(g')$, there exists a
unique morphism $\mu:T\to H(A')$ such that $H(f')\mu=H(g)t$. Now by the universal property of the pullback of $H(g)$ and
$H(f')$ there exists a unique morphism $\lambda:T\to H(K)$ such that $H(k)\lambda=t$ and $H(\alpha)\lambda=\mu$. If
there is also a morphism $\lambda':T\to H(k)$ in $\widehat{\mathcal{C}}$ such that $H(k)\lambda'=t$ and
$H(\alpha)\lambda'=\mu$, then it follows that $\lambda'=\lambda$, because $H(k)$ is a monomorphism. Hence $H(k)$ is the
kernel of the deflation $H(g'g)$ in $\widehat{\mathcal{C}}$, and so $H(k)$ is an inflation in $\widehat{\mathcal{C}}$.
Then by the choice of $\mathcal{E}_0$ it follows that the kernel-cokernel pair $K\stackrel{k}\rightarrowtail
B\stackrel{g'g}\twoheadrightarrow C'$ belongs to $\mathcal{E}_0$. This shows that $[E1]$ holds for $\mathcal{E}_0$, and
so $\mathcal{E}_0$ defines an exact structure on $\mathcal{C}$. 
\end{proof}

Now we are in a position to give our main result.

\begin{theorem} \label{t:maxex} Let $\mathcal{C}$ be an additive category, and let $(\widehat{\mathcal{C}},H)$ be its idempotent
completion. Consider the maximal exact structure $\mathcal{E}_{\rm max}$ in $\widehat{\mathcal{C}}$ given by the stable short 
exact sequences in $\widehat{\mathcal{C}}$. Then the stable short exact sequences define an exact structure on $\mathcal{C}$ 
if and only if $\mathcal{C}$ is closed under pushouts and pullbacks for $(\widehat{\mathcal{C}},\mathcal{E}_{\rm max})$. 
In this case, it is the maximal exact structure on $\mathcal{C}$.
\end{theorem}

\begin{proof} Stable short exact sequences define the maximal exact structure $\mathcal{E}_{\rm max}$ in $\widehat{\mathcal{C}}$ 
by \cite[Theorem~3.5]{C12}.

If the stable short exact sequences define an exact structure on $\mathcal{C}$, then 
$\mathcal{C}$ is closed under pushouts and pullbacks for $(\widehat{\mathcal{C}},\mathcal{E}_{\rm max})$ by Proposition \ref{p:char}.

Conversely, assume that $\mathcal{C}$ is closed under pushouts and pullbacks for $(\widehat{\mathcal{C}},\mathcal{E}_{\rm max})$. 
By \cite[Theorem~3.5]{C12}, the stable exact sequences define an exact structure on the (weakly)
idempotent complete category $\widehat{\mathcal{C}}$. Consider the distinguished class $\mathcal{E}_0$ of
kernel-cokernel pairs $A\rightarrowtail B\twoheadrightarrow C$ in $\mathcal{C}$ such that the induced kernel-cokernel
pair $H(A)\rightarrowtail H(B)\twoheadrightarrow H(C)$ is a stable exact sequence in $\widehat{\mathcal{C}}$.
Then $\mathcal{E}_0$ defines an exact structure on $\mathcal{C}$ by Proposition \ref{p:char}, and its conflations must be stable
short exact sequences. On the other hand, if $A\to B\to C$ is a stable short exact sequence in $\mathcal{C}$, then the
short exact sequence $H(A)\to H(B)\to H(C)$ is stable in $\widehat{\mathcal{C}}$ by Proposition \ref{p:main}, hence
$A\to B\to C$ is a conflation in $\mathcal{C}$ with respect to $\mathcal{E}_0$. Therefore, $\mathcal{E}_0$ coincides
with the class of stable short exact sequences. 

Clearly, this is the maximal exact structure on $\mathcal{C}$. 
\end{proof}

The stability condition from Theorem \ref{t:maxex} is usually easier to check in specific categories. In what follows we give situations 
when it does or does not hold. 

\begin{corollary} The stable short exact sequences define the maximal exact structure on the category of free modules.  
\end{corollary}

\begin{proof} It is well known that the idempotent completion of the category $\mathcal{F}$ of free modules 
is the category $\mathcal{P}$ of projective modules. Consider the usual functor $H:\mathcal{F}\to \mathcal{P}$ and the maximal exact 
structure $\mathcal{E}_{\rm max}$ on $\mathcal{P}$ given by the stable short exact sequences. 
If one starts with a stable short exact sequence in $\mathcal{P}$ of the form 
$H(A)\stackrel{H(i)}\rightarrowtail H(B)\stackrel{H(d)}\twoheadrightarrow H(C)$, 
there exists the pullback of $H(i)$ and any morphism $H(f):H(A)\to H(A')$. Then one has an induced stable short exact sequence $H(A')\to P\to H(C)$ 
\cite[Proposition~2.12]{Buhler}, which splits by the projectivity of $H(C)$. Since $\mathcal{P}$ is weakly idempotent complete, 
it follows that $P\cong H(A')\oplus H(C)$. Hence $P$ is free, and so the stability condition for pushouts 
in the sense of Definition \ref{d:stablePOPB} is checked. Dually, one has the stability condition for pullbacks. 
Hence the stable short exact sequences define the maximal exact structure on $\mathcal{F}$ by Theorem \ref{t:maxex}.
\end{proof}

\begin{example} \label{e:SR} \rm Consider the ring $R=k[x]/(x^2)$ for some field $k$, and denote $S=k_R$. Let $\mathcal{C}$ be 
the additive category $\langle S,R\oplus R\rangle$ generated by $S$ and $R\oplus R$. Then its idempotent completion $\widehat{\mathcal{C}}$ is 
the additive category $\langle S,R\rangle$ generated by $S$ and $R$, and then $\widehat{\mathcal{C}}={\rm Mod}(R)$. 
Then we have the following commutative diagram in $\widehat{\mathcal{C}}$:
\[\SelectTips{cm}{}
\xymatrix{
S\oplus S \ar[d]_{(0,1)} \ar@{>->}[r]^-{(i,i)} & R\oplus R \ar[d]^{(p,1)} \ar[r] & S\oplus S \ar@{=}[d] \\ 
0\oplus S \ar@{>->}[r]^-{(0,i)} & S\oplus R \ar[r] & S\oplus S 
}\]
in which $i$ is the inclusion, $p$ is the projection map with $pi=0$, and the remaining morphisms are the canonical ones. 
Then the left square is a pushout, and the upper kernel-cokernel pair is stable, since $\widehat{\mathcal{C}}={\rm Mod}(R)$. 
Also, the first and the last objects of the lower kernel-cokernel pair are objects in $\mathcal{C}$. But the middle object $S\oplus R$ 
does not belong to $\mathcal{C}$. This shows that the stability condition for pushouts in the sense of Definition \ref{d:stablePOPB} does not hold, 
hence the stable short exact sequences do not define an exact structure on $\mathcal{C}$ by Theorem \ref{t:maxex}.
\end{example}

We end with two applications to categories of chain complexes and projective spectra. 

\begin{corollary} Let $\mathcal{C}$ be an additive category, and let $(\widehat{\mathcal{C}},H)$ be its idempotent
completion. Consider the maximal exact structure $\mathcal{E}_{\rm max}$ in $\widehat{\mathcal{C}}$ given by the stable short 
exact sequences in $\widehat{\mathcal{C}}$, and assume that $\mathcal{C}$ is closed under pushouts and pullbacks for 
$(\widehat{\mathcal{C}},\mathcal{E}_{\rm max})$. Then the maximal exact structure on the category $\mathbf{Ch}(\mathcal{C})$ of
chain complexes over $\mathcal{C}$ consists of the short exact sequences which are stable short exact sequences in $\mathcal{C}$ in each degree.
\end{corollary}

\begin{proof} If the additive category $\mathcal{C}$ has an exact structure $\mathcal{E}$, then
the additive category $\mathbf{Ch}(\mathcal{C})$ also has an exact structure $\mathbf{Ch}(\mathcal{E})$ whose
conflations are the short exact sequences which are conflations from $\mathcal{E}$ in each degree (e.g., see
\cite[Lemma~9.1]{Buhler}). In particular, if $\mathcal{E}_{\rm max}^{\mathcal{C}}$ is the maximal exact structure on
$\mathcal{C}$, then $\mathbf{Ch}(\mathcal{E}_{\rm max}^{\mathcal{C}})$ is an exact
structure on $\mathbf{Ch}(\mathcal{C})$, included in the maximal exact structure $\mathcal{E}_{\rm
max}^{\mathbf{Ch}(\mathcal{C})}$ on $\mathbf{Ch}(\mathcal{C})$. 

Now let $A\overset{i}\rightarrowtail B\overset{d}\twoheadrightarrow C$ be a conflation from $\mathcal{E}_{\rm
max}^{\mathbf{Ch}(\mathcal{C})}$, that is, a stable short exact sequence in $\mathbf{Ch}(\mathcal{C})$ by Theorem
\ref{t:maxex}. Set a degree $n$, let $X$ be an object of $\mathcal{C}$ and let $\alpha:X\to C^n$ be a morphism in
$\mathcal{C}$. Then $X$ can be viewed as a chain complex $C'$ concentrated in degree $n$, and we may define a chain map
$h:C'\to C$ by $h^n=\alpha$ and $h^m=0$ for every $m\neq n$. By Lemma \ref{l:RW}, the pullback of $d$ and $h$ yields
the following commutative diagram in $\mathbf{Ch}(\mathcal{C})$:
\[\SelectTips{cm}{}
\xymatrix{
A \ar@{=}[d] \ar@{>->}[r]^{i'} & B' \ar[d]_g \ar@{->>}[r]^{d'} & C' \ar[d]^h \\ 
A \ar@{>->}[r]^i & B \ar@{->>}[r]^d & C  
}\] Then $A\overset{i'}\rightarrowtail B'\overset{d'}\twoheadrightarrow C'$ is a conflation from $\mathcal{E}_{\rm
max}^{\mathbf{Ch}(\mathcal{C})}$ by Theorem \ref{t:maxex}. Then $A^n\overset{i'^n}\rightarrowtail
B'^n\overset{d'^n}\twoheadrightarrow C'^n$ is a kernel-cokernel pair in $\mathcal{C}$, and so
$A^n\overset{i^n}\rightarrowtail B^n\overset{d^n}\twoheadrightarrow C^n$ is a stable short exact sequence in
$\mathcal{C}$. Thus $A\overset{i}\rightarrowtail B\overset{d}\twoheadrightarrow C$ is a conflation from
$\mathbf{Ch}(\mathcal{E}_{\rm max}^{\mathcal{C}})$, which shows that $\mathbf{Ch}(\mathcal{E}_{\rm
max}^{\mathcal{C}})=\mathcal{E}_{\rm max}^{\mathbf{Ch}(\mathcal{C})}$.
\end{proof}

Following \cite[Definition~7.1]{FS} a \emph{projective spectrum} $X=(X_n,X_m^n)$ with values in a category
$\mathcal{C}$ consists of a sequence $(X_n)_{n\in \mathbb{N}}$ of objects of $\mathcal{C}$ and morphisms $X_m^n:X_m\to
X_n$ in $\mathcal{C}$ defined for $n\leq m$ such that $X_n^n=1_{X_n}$ for every $n\in \mathbb{N}$, and 
$X_n^k\circ X_m^n=X_m^k$ for $k\leq n\leq m$. A \emph{morphism of projective spectra} $f:X\to Y$ between two projective
spectra $X=(X_n,X_m^n)$ and $Y=(Y_n,Y_m^n)$ consists of a sequence $(f_n)_{n\in \mathbb{N}}$ of morphisms $f_n:X_n\to
Y_n$ in $\mathcal{C}$ such that $f_n\circ X_m^n=Y_m^n\circ f_m$ for $n\leq m$. If $\mathcal{C}$ is additive, then so is
the category $\mathbf{P}(\mathcal{C})$ of projective spectra with values in $\mathcal{C}$.

\begin{corollary} Let $\mathcal{C}$ be an additive category, and let $(\widehat{\mathcal{C}},H)$ be its idempotent
completion. Consider the maximal exact structure $\mathcal{E}_{\rm max}$ in $\widehat{\mathcal{C}}$ given by the stable short 
exact sequences in $\widehat{\mathcal{C}}$, and assume that $\mathcal{C}$ is closed under pushouts and pullbacks for 
$(\widehat{\mathcal{C}},\mathcal{E}_{\rm max})$. Then the maximal exact structure on the category $\mathbf{P}(\mathcal{C})$ 
consists of the short exact sequences $A\overset{i}\to B\overset{d}\to C$ for which 
$A_n\overset{i_n}\rightarrowtail B_n\overset{d_n}\twoheadrightarrow C_n$ are stable short exact sequences in
$\mathcal{C}$ for every $n\in \mathbb{N}$.
\end{corollary}

\begin{proof} It is similar to the proof of \cite[Corollary~7.4]{FS}, using Theorem \ref{t:maxex}.
\end{proof}

\end{document}